\documentclass[reqno,11pt]{amsart}
\usepackage{amsmath,amsthm,amssymb,amsfonts,xcolor,cite}
\usepackage{a4wide}
\usepackage[colorlinks]{hyperref}
\usepackage[showonlyrefs]{mathtools}
\mathtoolsset{showonlyrefs=true}
\usepackage{mathrsfs}

\newtheorem{theorem}{Theorem}[section]
\newtheorem{proposition}[theorem]{Proposition}
\newtheorem{corollary}[theorem]{Corollary}
\newtheorem{lemma}[theorem]{Lemma}

\newtheorem{remark}[theorem]{Remark}
\newtheorem{definition}[theorem]{Definition}

\newcommand{\R}{\mathbb R}

\def\eps{\varepsilon}

\def\S{\mathbb{S}}
\def\les{\lesssim}

\newcommand{\dHaus}{\,d\mathcal{H}^{n-1}}

\numberwithin{equation}{section}

\author{Michael Goldman}
\address{LAMA, Univ Paris Est Creteil, Univ Gustave Eiffel, UPEM, CNRS, F-94010 Cr\'eteil, France}
\email{michael.goldman@cnrs.fr}

\author{Matteo Novaga}
\address{Dipartimento di Matematica, Universit\`a di Pisa, Pisa, Italy}
\email{matteo.novaga@unipi.it}

\author{Berardo Ruffini}
\address{Dipartimento di Matematica, Universit\`a di Bologna, Bologna, Italy}
\email{berardo.ruffini@unibo.it}

\title[Second variation of the fractional torsion energy]{Second
variation of the fractional torsion energy\\ and linear 
stability of the ball}


\begin{document}

\begin{abstract} 
       We
compute the second shape derivative of the fractional torsion. 
It is an explicit quadratic form on the boundary, which
involves a Dirichlet-to-Neumann operator for half-harmonic functions
blowing up at the boundary. We deduce an explicit formula for volume preserving variations on the ball and show 
 that as in the local case, the shape Hessian is negative and coercive in \(H^{1/2}\) modulo
translations.
\end{abstract}

\maketitle

\section{Introduction}
\label{sec:intro}
The aim of this paper is to show, in arguably the simplest meaningful
situation, how second order shape derivatives can be computed for shape
optimization problems involving operators of fractional order. To this end we
compute the second shape derivative of the fractional torsion energy
for smooth domains and the exponent $s=1/2$. While second order shape
derivatives are well understood in the local setting, see e.g.\ \cite{HP,BDV},
and first order derivatives for fractional problems have received considerable
attention, see e.g.\ \cite{DGV,DFW,DFWPohozaev,muratov2022conducting},
second order derivatives in the fractional setting have been less studied.
We expect the approach developed here to be useful for related problems,
such as fractional capacities and first Dirichlet eigenvalues, and for
other exponents $s\in(0,1)$.

Our initial motivation was quantitative stability of the ball. For problems
in which the ball is a volume-constrained optimizer (a minimizer or a
maximizer, depending on the problem), computing the second variation and
establishing its coercivity are essential steps in carrying out the program
of \cite{BDV}, namely a proof of sharp quantitative stability through
the study of nearly spherical sets and a selection principle. We were
indeed on this path when we found, in \cite{GNRfractional}, a much shorter
and less technical proof of quantitative stability, adapting the strategy
of \cite{GMR}. We nevertheless decided to write down a rigorous derivation of
the second variation, since it may be useful for other purposes. With this in
mind we also prove its coercivity at the ball under volume-preserving
variations.

Following the full strategy of \cite{BDV} would require further results.
First, one would need weaker regularity assumptions on the domains, or a
regularity bootstrap for the nearly spherical sets arising in the
contradiction argument; see \cite{Lyu2025,BarriosRosOtonWeidner2025}.
Second, one would need continuity estimates for the second variation in
the relevant Sobolev norm, here $H^{1/2}$. This might require the framework
of \cite{Grubb1996,Grubb2014,Grubb2015}, instead of the regularity results of
\cite{ARO} used here. Finally, the selection principle would require an
adaptation of the regularity theory for fractional Bernoulli free boundary
problems developed in \cite{DeSilvaSavinSire2014}.

Let \(\Omega\subset\R^n\), \(n\ge2\), be a bounded open set with
\(C^\infty\) boundary. The \emph{fractional torsion function}
\(u_\Omega\) is the solution of
\begin{equation}\label{eq:torsion}
 (-\Delta)^{1/2}u_\Omega=1\ \text{ in }\Omega,\qquad u_\Omega=0\ \text{ in }\R^n\setminus\Omega,
\end{equation}
where \((-\Delta)^{1/2}\) is the fractional Laplacian with Fourier symbol
\(|\xi|\),
\begin{equation}\label{eq:normalization}
 (-\Delta)^{1/2}v(x)=c_n\,\mathrm{P.V.}\!\int_{\R^n}\frac{v(x)-v(y)}{|x-y|^{n+1}}\,dy,
 \qquad
 c_n=\frac{\Gamma((n+1)/2)}{\pi^{(n+1)/2}} ,
\end{equation}
and the \emph{fractional torsional rigidity}, or torsion energy, of
\(\Omega\) is 
\[T(\Omega)=\int_\Omega u_\Omega.\] 

The fractional  Saint-Venant
inequality states that balls are the (unique) maximizers of the fractional torsion under volume constraint,
 see e.g. \cite{BCV,GNRfractional}. It is then a natural question to compute the corresponding first and 
 second order shape derivatives. Let us first recall what is known for the first order derivatives from the works \cite{DGV,DFW,DFWPohozaev,DS,muratov2022conducting}. For this we need to introduce some notation.

Problem \eqref{eq:torsion} is understood in the energy space
\begin{equation}\label{eq:energy}
 \begin{split}
 V&=\{v\in H^{1/2}(\R^n):v=0\text{ a.e. in }\R^n\setminus\Omega\},\\
 a(v_1,v_2)&=\frac{c_n}{2}\iint_{\R^{2n}}
 \frac{(v_1(y)-v_1(z))(v_2(y)-v_2(z))}{|y-z|^{n+1}}\,dy\,dz ,
 \end{split}
\end{equation}
so that \(a(v_1,v_2)=\int_{\R^n}v_1(-\Delta)^{1/2}v_2\) for smooth
functions. Then \(u_\Omega\) is the unique element of \(V\) with \(a(u_\Omega,v)=\int_\Omega v\) for
every \(v\in V\), and
\begin{equation}\label{eq:T-max}
 T(\Omega)=\max_{v\in V}\Big\{2\int_\Omega v-a(v,v)\Big\},
\end{equation}
the maximum being attained at \(v=u_\Omega\).

We write \(\nu\) for the outer unit normal to \(\partial\Omega\),
\(H=\operatorname{div}_{\partial\Omega}\nu\) for the mean curvature  and
\(d\) for the distance to \(\partial\Omega\). We fix \(r_0>0\) such that
\((x,r)\mapsto x-r\nu(x)\) is a diffeomorphism from
\(\partial\Omega\times(-2r_0,2r_0)\) onto a neighbourhood of
\(\partial\Omega\), and a function \(\delta\in C^\infty(\overline\Omega)\),
positive in \(\Omega\), with \(\delta=d\) on 
\(\{y\in\Omega:d(y)<r_0\}\). Thus \(d(x-r\nu(x))=\delta(x-r\nu(x))=r\) for
\(x\in\partial\Omega\) and \(0<r<r_0\).

The torsion function behaves like \(\sqrt d\) at the
boundary~\cite{RSJMPA}: by Lemma~\ref{lem:regularity} below,
\(u_\Omega=\sqrt\delta\,U\) with \(U\in C^\infty(\overline\Omega)\). The
function \(U\) depends on the choice of \(\delta\) away from  \(\partial\Omega\),
but its behaviour near \(\partial\Omega\) does not, since there
\(U(x-r\nu(x))=r^{-1/2}u_\Omega(x-r\nu(x))\). The torsion function enters
our formulas only through two boundary quantities: the \emph{weighted trace} \(\psi\) of
\(u_\Omega\) and its first normal coefficient \(\kappa\), defined for
\(x\in\partial\Omega\) by
\begin{equation}\label{eq:coefficients}
 \psi(x)=\lim_{r\downarrow0}\frac{u_\Omega(x-r\nu(x))}{\sqrt r}=U(x),\qquad
 \kappa(x)=\left.\partial_r U(x-r\nu(x))\right|_{r=0},
\end{equation}
so that uniformly in \(x\in\partial\Omega\)
\begin{equation}\label{eq:u-expansion}
 u_\Omega(x-r\nu(x))=\psi(x)\sqrt r+\kappa(x)\,r^{3/2}+O(r^{5/2}).
\end{equation}

In terms of these coefficients, the first variation of \(T\) is as follows.
Let \(\Phi_t\) be  any smooth family of diffeomorphisms of \(\R^n\) with
\(\Phi_0=\mathrm{Id}\), let \(X=\partial_t\Phi_t|_{t=0}\) and
\(\varphi=X\cdot\nu\) on \(\partial\Omega\). Let $\Omega_t=\Phi_t(\Omega)$. Then
\begin{equation}\label{eq:hadamard}
 \frac{d}{dt}T(\Omega_t)\Big|_{t=0}
 =c_*\int_{\partial\Omega}\psi^2\varphi\dHaus,
 \qquad c_*:=\Gamma(3/2)^2=\frac\pi4 .
\end{equation}
As in \cite{DFWPohozaev} where the fractional Dirichlet eigenvalue is considered,  
we recover the first variation formula \eqref{eq:hadamard} through a short application of
 the generalized fractional Pohozaev identity of \cite{DFWPohozaev}; see Remark~\ref{rem:first-variation}.

We now turn to the second order derivative. The difficulty comes from understanding the PDE
 which is solved by the Eulerian derivative $q=\partial_t u_{\Omega_t}|_{t=0}$.
  Indeed, by differentiating the PDE solved by $ u_{\Omega_t}$ we can guess that it is $1/2-$harmonic
   in $\Omega$ and equal to zero outside $\Omega$. What makes it both non-trivial and non-standard is that   
 it  blows up like the inverse
square root of the distance to the boundary. In particular, it lies outside the energy space in which
\eqref{eq:torsion} is posed. These are "large solutions " in the language of \cite{RW,Aba}. 
One of the main points of the analysis is thus to understand a certain
Dirichlet-to-Neumann-type operator \(\mathcal D_\Omega\), defined
in \eqref{eq:D} below, which acts on the boundary data of these \emph{large}
\(\tfrac12\)-harmonic functions. Here we will 
only need an elementary construction of these
functions and of \(\mathcal D_\Omega\) (Proposition~\ref{prop:large}) which we now describe.
 For every \(f\in C^\infty(\partial\Omega)\) we prove that 
there is a unique function \(Q_f\), vanishing outside \(\Omega\) and with
\(\sqrt\delta\,Q_f\in C^\infty(\overline\Omega)\), such that
\begin{equation}\label{eq:large}
 (-\Delta)^{1/2}Q_f=0\ \text{ in }\Omega,\qquad
 (\sqrt\delta\,Q_f)|_{\partial\Omega}=f.
\end{equation}
 Here the equation holds in the sense of
distributions: \(Q_f\) is integrable, but it blows up like \(d^{-1/2}\) at
\(\partial\Omega\) and is not in \(L^2\) unless \(f=0\). We set, for \(x\in\partial\Omega\),
\begin{equation}\label{eq:D}
 \mathcal D_\Omega f(x)
 =\left.\partial_r\big(\sqrt r\,Q_f(x-r\nu(x))\big)\right|_{r=0},
\end{equation}
that is, \(Q_f(x-r\nu(x))=f(x)\,r^{-1/2}+\mathcal D_\Omega f(x)\sqrt
r+O(r^{3/2})\). The construction of \(\mathcal D_\Omega\) in
Proposition~\ref{prop:large} only uses the regularity result Lemma~\ref{lem:regularity}.

We now give the main statement.  We will limit ourselves to the computation in the case of normal graphs.
 This is very common in the study of  shape derivatives, see \cite{HP}.
Fix \(\varphi\in C^\infty(\partial\Omega)\). For $x\in \partial \Omega$ and $|r|\le r_0/2$ we  set
\begin{equation}\label{eq:normal-extension}
 X(x-r\nu(x))=\varphi(x)\nu(x)
\end{equation}
which we extend to a smooth vector field $X$ compactly supported in $\{d<r_0\}$. We then set 
\[
\Phi_t(x)=x+tX(x),
\]
which is a diffeomorphism for $|t|$ small enough. We write  
\[
 \Omega_t=\Phi_t(\Omega),\qquad
 \partial\Omega_t=\{x+t\varphi(x)\nu(x):x\in\partial\Omega\}.
\]

\begin{theorem}\label{thm:hessian}
Let \(\Omega\subset\R^n\) be a bounded open set with
\(C^\infty\) boundary, and let \(\varphi\in
C^\infty(\partial\Omega)\). The function \(t\mapsto
T(\Omega_{t})\) is of class \(C^2\) near
\(0\), and
\begin{equation}\label{eq:hessian}
 \frac{d^2}{d t^2}
 T(\Omega_{t})\Big|_{t=0}
 =\mathscr H_\Omega(\varphi)
 :=c_*\int_{\partial\Omega}
 \Big\{\psi\varphi\,\mathcal D_\Omega(\psi\varphi)
 +\big(H\psi^2-3\psi\kappa\big)\varphi^2\Big\}\dHaus .
\end{equation}
\end{theorem}

The natural variable in \eqref{eq:hessian} is the weighted displacement
\(\psi\varphi\). The term \(H\psi^2\) comes from the variation of the area
element, and the term \(-3\psi\kappa\) from following the moving boundary
point in the expansion \eqref{eq:u-expansion} (see
\eqref{eq:formal-psi-dot} below). 

\smallskip

In the last part of the paper we compute the second variation of the unit ball for volume constrained (to first order) 
variations.
We show that this Hessian is  negative and coercive in
\(H^{1/2}(\S^{n-1})\) modulo translations. This result might appear surprising at first sight since this coercivity
 is in the same norm as for the local case, see \cite{BDV}. This comes from
  the fact that as in the local case, the operator involved (here $\mathcal D_\Omega(\psi\varphi)$)
   is a Dirichlet-to-Neumann map which is thus of order one.

Let \(\mathcal P\) denote harmonic extension to \(B=B_1\),
\(\Lambda=\partial_\nu\mathcal P\) the classical Dirichlet-to-Neumann
operator, and \(\varphi_\ell\) the projection onto spherical harmonics of
degree \(\ell\). With \(\|\cdot\|_2\) denoting the \(L^2(\S^{n-1})\) norm,
set
\begin{equation}\label{eq:H12R-norm}
 \|\varphi\|_{H^{1/2}_*}^2
 :=\sum_{\ell\ge0}(1+\ell)\|\varphi_\ell\|_2^2
 =\|\varphi\|_2^2+\int_B|\nabla\mathcal P\varphi|^2\,dy,
\end{equation}
which is equivalent to the usual \(H^{1/2}(\S^{n-1})\) norm.  Let $\psi$ be the weighted trace of $u_B$ and set \(\lambda_*=c_*\psi^2\).
 For any $\varphi\in C^\infty(\S^{n-1})$ with 
\begin{equation}\label{eq:ball-volume-tangent}
 \int_{\S^{n-1}}\varphi\dHaus=0,
\end{equation}
define the quadratic form 
\begin{equation}\label{eq:constrained}
\mathcal Q(\varphi)
 :=\lambda_*\int_{\S^{n-1}}
 \big(\varphi\Lambda\varphi-\varphi^2\big)\dHaus
 =\lambda_*\sum_{\ell\ge0}(\ell-1)\|\varphi_\ell\|_2^2 .
\end{equation}
For  $h\in C^\infty(\S^{n-1})$, we set 
$\Omega_h=\{ \rho \theta \ : \ \rho\in[0,1+h(\theta)], \ \theta\in \S^{n-1}\}$.   
\begin{theorem}\label{thm:ball-coercivity}
 For any smooth volume-preserving path \(\Omega_{h(t)}\) with $h(0)=0$, setting  
  \(\varphi=\dot h(0)\) (which satisfies \eqref{eq:ball-volume-tangent} 
  by the volume preserving property),
   we have 
\begin{equation}\label{Hessconstraint}
 -\frac{d^2}{dt^2}T(\Omega_{h(t)})\Big|_{t=0}=\mathcal Q(\varphi).
\end{equation}
If moreover $\varphi$ satisfies
\begin{equation}\label{eq:ball-centering}
 \int_{\S^{n-1}}\varphi\,\nu\dHaus=0,
\end{equation}
then
\begin{equation}\label{eq:ball-coercivity}
 \frac{\lambda_*}{3}\,\|\varphi\|_{H^{1/2}_*}^2
 \le \mathcal Q(\varphi)\le\lambda_*\,\|\varphi\|_{H^{1/2}_*}^2 .
\end{equation}
The constant \(1/3\) is attained exactly on spherical harmonics of degree
two. 
\end{theorem}
Conditions \eqref{eq:ball-volume-tangent} and \eqref{eq:ball-centering} correspond to (first order) 
volume preserving and translation modes.

\subsection{A formal derivation of the second variation}
\label{ssec:formal}

The following computation explains the structure of \eqref{eq:hessian}. It
assumes that all the objects involved are smooth in \(t\) and that all
expansions may be differentiated; it is not used in the proof.

For \(x\in\partial\Omega\) set  \(x_t=\Phi_t(x)=x+t\varphi(x)\nu(x)\), write \(\nu_t\) for the outer normal to \(\Omega_t\), and
\(u_t=u_{\Omega_t}\), \(\psi_t\), \(\kappa_t\) for the torsion function of \(\Omega_t\)
and its boundary coefficients. The normal velocity of \(\partial\Omega_t\)
at \(x_t\) is \(\varphi_t(x_t)=\varphi(x)\nu(x)\cdot\nu_t(x_t)\).

\emph{First variation at time \(t\).} By \eqref{eq:hadamard} applied at
\(\Omega_t\), we have 
\(\frac{d}{dt}T(\Omega_t)=c_*\int_{\partial\Omega_t}\psi_t^2\varphi_t\dHaus\).

\emph{Differentiation of the boundary integral.} Differentiating once more
using the standard formula for boundary integrals, see \cite{HP}, and
noting that $\partial_t\varphi_t(x_t)|_{t=0}=0$, we obtain
\begin{equation}\label{eq:formal-step2}
 \frac{d^2}{dt^2}T(\Omega_t)\Big|_{t=0}
 =c_*\int_{\partial\Omega}\big\{2\psi\,\dot\psi\,\varphi+H\psi^2\varphi^2\big\}\dHaus,
 \qquad \dot\psi(x):=\frac{d}{dt}\psi_t(x_t)\Big|_{t=0},
\end{equation}
and everything reduces to the computation of \(\dot\psi\).

\emph{Differentiation of the weighted trace.} Recall that  \(q=\partial_tu_t|_{t=0}\). We have
\[
 \psi_t(x_t)=\lim_{r\to0}r^{-1/2}u_t(x_t-r\nu_t(x_t)),\qquad
 \partial_t(x_t-r\nu_t(x_t))\big|_{t=0}
 =\varphi\nu+r\nabla_{\partial\Omega}\varphi.
\]
Differentiating at fixed \(r\) gives
\[
 \frac{d}{dt}\Big[\frac{u_t(x_t-r\nu_t(x_t))}{\sqrt r}\Big]_{t=0}
 =\frac1{\sqrt r}\Big\{q(x-r\nu)+\varphi\,\nu\cdot\nabla u_\Omega(x-r\nu)
 +r\,\nabla_{\partial\Omega}\varphi\cdot\nabla u_\Omega(x-r\nu)\Big\}.
\]
By \eqref{eq:u-expansion}, \(\nu\cdot\nabla u_\Omega(x-r\nu)
=-\frac{\psi}{2\sqrt r}-\frac32\kappa\sqrt r+O(r^{3/2})\), while the
tangential part of \(\nabla u_\Omega(x-r\nu)\) is \(O(\sqrt r)\). We shall see
below that
\begin{equation}\label{eq:formal-q}
 q(x-r\nu(x))=\frac{\psi\varphi}{2\sqrt r}+B_\varphi\sqrt r+O(r^{3/2})
\end{equation}
for some function \(B_\varphi\). The two terms of order \(r^{-1/2}\) cancel,
and letting \(r\to0\) gives
\begin{equation}\label{eq:formal-psi-dot}
 \dot\psi=B_\varphi-\tfrac32\kappa\varphi .
\end{equation}

\emph{Identification of \(B_\varphi\).} For \(\chi\in C_c^\infty(\Omega)\)
and \(|t|\) small, \(\int u_t(-\Delta)^{1/2}\chi=\int\chi\); differentiating
in \(t\) shows that \(q\) is \(\tfrac12\)-harmonic in \(\Omega\). It also
vanishes outside \(\overline\Omega\). Let  \(w=\partial_t(u_t\circ
\Phi_t)|_{t=0}\) be the material derivative. By the chain rule we have 
 \[q=w-X\cdot\nabla u_\Omega.\] 
The
function \(w\) vanishes like \(\sqrt d\) on \(\partial\Omega\), whereas
\(X\cdot\nabla u_\Omega=-\frac{\psi\varphi}{2}d^{-1/2}+O(\sqrt d)\) by
\eqref{eq:u-expansion}. Thus \(q\) is a large \(\tfrac12\)-harmonic function
with weighted trace \(\psi\varphi/2\). By uniqueness in \eqref{eq:large},
\(q=Q_{\psi\varphi/2}\), which gives \eqref{eq:formal-q} with
\begin{equation}\label{eq:Bphi}
 B_\varphi=\tfrac12\,\mathcal D_\Omega(\psi\varphi).
\end{equation}
Inserting \eqref{eq:formal-psi-dot} and \eqref{eq:Bphi} into
\eqref{eq:formal-step2} gives  \eqref{eq:hessian}.

Two steps of this computation are formal. The identification \(q=Q_{\psi
\varphi/2}\) requires some care, because \(q\) is not in \(L^2\) (as soon as
\(\varphi\not\equiv0\)) and therefore lies outside the variational framework
in which \(u_t\) is defined; this is done rigorously on the fixed domain in
Proposition~\ref{prop:q}. The genuinely formal step is the derivation of
\eqref{eq:formal-psi-dot}, which interchanges \(\partial_t\) and
\(\lim_{r\to0}\). Justifying it would require expansions of \(u_t\) and of the
large solutions on \(\Omega_t\), together with their \(t\)-derivatives, that
are uniform in \(t\). The proof below avoids this interchange.

\subsection{Strategy of the proof}
\label{ssec:strategy}

The rigorous proof computes \(\frac{d^2}{dt^2}T(\Omega_t)\) without ever
differentiating \(\psi_t\). All steps but one take place on the fixed domain
\(\Omega\).

We pull the problem back to \(\Omega\) by the diffeomorphisms
\(\Phi_t=\mathrm{Id}+tX_\varphi\) of \eqref{eq:normal-extension}, whose
velocity is constant along the normals near \(\partial\Omega\). The
pulled-back states \(v_t=u_t\circ\Phi_t\) solve
\(a_t(v_t,v)=\ell_t(v)\) in the fixed space \(V\), where \(a_t\) and
\(\ell_t\) are obtained from (recall \eqref{eq:energy}) \(a\) and \(\int v\) by a change of variables.
From \eqref{eq:T-max} we get the exact identity
\(T(\Omega_t)=\mathcal G(t)+a_t(v_t-u_\Omega,v_t-u_\Omega)\) with
\(\mathcal G(t)=2\ell_t(u_\Omega)-a_t(u_\Omega,u_\Omega)\), and therefore
\[
 \frac{d^2}{dt^2}T(\Omega_t)\Big|_{t=0}=\mathcal G''(0)+2a(w,w),
\]
where we recall that \(w=\partial_tv_t|_{t=0}\) is the material derivative, see Proposition~\ref{prop:bulk}.
 The aim is then to rewrite the right-hand side as a boundary integral. We start with the term $a(w,w)$.
 For this we show in Proposition~\ref{prop:q}
that \(q=w-X\cdot\nabla u_\Omega=Q_{\psi\varphi/2}\). To prove this equality, we first show that  the
\(d^{-1/2}\) singularities of \(q\) and of \(Q_{\psi\varphi/2}\) cancel, so
that their difference lies in \(V\) and can be treated by a more standard energy argument.
This shows that \(w=\sqrt\delta\,W\) with \(W\) smooth and gives the boundary
value of \(W\) (Lemma~\ref{lem:material-coefficient}). A bilinear Pohozaev
identity from \cite{DFWPohozaev} then turns \(a(w,w)\)
into a volume integral plus a boundary integral
(Corollary~\ref{cor:poho-energy}).

It thus  remains to compute \(\mathcal G''(0)\). The same Pohozaev identity, applied
on \(\Omega_t\) to the transported function \(u_\Omega\circ\Phi_t^{-1}\) and pulled
back to \(\Omega\), gives
\[
 \mathcal G'(t)=2\int_\Omega(L_tu_\Omega-1)\,\hat X_t\cdot\nabla u_\Omega\,J_t+\mathcal B(t),
\]
where \(L_t\) is the pull-back of \((-\Delta)^{1/2}\) by \(\Phi_t\),
\(J_t\) the Jacobian, \(\hat X_t\) a transported velocity, and \(\mathcal
B(t)\) an explicit boundary integral (Lemma~\ref{lem:G-prime}). Only the
kernel of the operator depends on \(t\) here. Since \(L_0u_\Omega=1\), computing
\(\mathcal G''(0)\) amounts to passing to the limit in \((L_tu_\Omega-1)/t\), tested
against \(X\cdot\nabla u_\Omega\). The distributional limit is easy to identify, but
\(X\cdot\nabla u_\Omega\) behaves like \(d^{-1/2}\), so we also need the uniform bound, see Lemma~\ref{lem:lipschitz},
\[
 \|L_tu_\Omega-L_0u_\Omega\|_{L^\infty(\Omega)}\le C|t|,
\]
 which is the main technical estimate of the
paper and which may be seen as a simplified version of related estimates from \cite{ARO}. A naive Taylor expansion only gives a bound of order
\(|t|\,d^{-1/2}\). The proof uses instead that, after the kernel of \(L_t\)
is frozen at the evaluation point, the difference \(L_t-L_0\) becomes an
operator \(\mathscr P_b\) with an even angular density \(b\), and that such
operators annihilate \(\sqrt{\mathsf l_+}\) for every affine function
\(\mathsf l\) (Lemma~\ref{lem:frozen-angular}). The volume integral in
\(\mathcal G''(0)\) then cancels exactly the one produced by the Pohozaev
identity, and only boundary terms remain. 

\subsection*{Organisation of the paper}
Section~\ref{sec:prelim} collects the results on nonlocal equations that we
use: boundary regularity, large solutions and the operator
\(\mathcal D_\Omega\), and the estimate for frozen angular operators.
Section~\ref{sec:proof-hessian} contains the proof of
Theorem~\ref{thm:hessian}, and Section~\ref{sec:ball} the proof of
Theorem~\ref{thm:ball-coercivity}.

\section{Preliminaries on nonlocal equations}
\label{sec:prelim}

Throughout this section \(\Omega\subset\R^n\) is a bounded open set with
\(C^\infty\) boundary, and \(r_0\), \(\delta\), \(V\), \(a\) are as in the
introduction. Functions defined in \(\Omega\) are extended by zero to
\(\R^n\) without further mention. For a function \(g\in L^1(\R^n)\) with
compact support, \((-\Delta)^{1/2}g\) is the distribution
\(\chi\mapsto\int g\,(-\Delta)^{1/2}\chi\), which is well defined since
\((-\Delta)^{1/2}\chi\) is smooth and decays like \(|x|^{-n-1}\) for
\(\chi\in C_c^\infty(\R^n)\).

\subsection{The energy space}

Since \(\partial\Omega\) is smooth, \(C_c^\infty(\Omega)\) is dense in
\(V\); this holds for every bounded open set with continuous boundary,
see~\cite{FSV}. Fix a bounded set \(\mathcal A\subset\R^n\setminus\Omega\) of
positive measure at positive distance from \(\Omega\). Since every \(v\in V\)
vanishes on \(\mathcal A\), keeping only the blocks \(\Omega\times\mathcal A\)
and \(\mathcal A\times\Omega\) in
the definition of \(a\) gives
\begin{equation}\label{eq:L2-by-energy}
 a(v,v)\ \ge\ c_n\int_\Omega|v(y)|^2
 \Big(\int_{\mathcal A}\frac{dz}{|y-z|^{n+1}}\Big)dy
 \ \ge\ c^{-1}\|v\|_{L^2(\Omega)}^2 ,
\end{equation}
with \(c\) depending on \(\Omega\) and \(\mathcal A\). Hence
\(\|v\|_V:=a(v,v)^{1/2}\) is a norm on \(V\), equivalent to the
\(H^{1/2}(\R^n)\) norm, and \(a\) is coercive on \(V\).

We shall use the comparison principle in the following form: if \(v\in V\)
and \(a(v,\phi)\ge0\) for every \(\phi\in V\) with \(\phi\ge0\), then
\(v\ge0\). Indeed, testing with \(\phi=v^-\) and using
\(a(v^+,v^-)\le0\), which holds because
\((v^+(y)-v^+(z))(v^-(y)-v^-(z))\le0\) pointwise, we get
\(0\le a(v,v^-)=a(v^+,v^-)-a(v^-,v^-)\le-a(v^-,v^-)\), so \(v^-=0\). In
particular \(u\ge0\), and comparison with the torsion function of a ball
containing \(\Omega\) shows that \(u\) is bounded.

\subsection{Boundary regularity}

Lemma~\ref{lem:regularity} below is the boundary regularity input of the
paper. It follows from \cite[Theorem~1.4]{ARO} and  \cite[Theorem~2.2 and Corollary~2.3]{ARO}.

\begin{lemma}\label{lem:regularity}
For every \(v\in V\), \(f\in C^\infty(\overline\Omega)\) and
\(A\in C^\infty(\overline\Omega)\),
\begin{equation}\label{eq:regularity}
 \begin{split}
 &(-\Delta)^{1/2}v=f\ \text{ in }\Omega
 \ \Longrightarrow\ v/\sqrt\delta\in C^\infty(\overline\Omega),\\
 &(-\Delta)^{1/2}(\sqrt\delta\,A)
                      \in C^\infty(\overline\Omega).
 \end{split}
\end{equation}
\end{lemma}

We shall also need the following energy membership property which essentially follows from the well-known fact that the function \(x\mapsto(x_n)_+^{1/2}\) belongs to
\(H^{1/2}_{\mathrm{loc}}(\R^n)\), see e.g. \cite[Lemmas~3.1 and~3.2]{IMS2016}, together with a localization argument.
\begin{lemma}\label{lem:energy-membership}
If \(A\in C^1(\overline\Omega)\), then \(\sqrt\delta\,A\in V\).
\end{lemma}


\subsection{Singular lifting}

To construct large \(\tfrac12\)-harmonic functions, we first write down an
explicit function with the prescribed singularity \(f\,d^{-1/2}\) and check
that its image under \((-\Delta)^{1/2}\) is smooth up to the boundary. The
remaining correction is then found in the energy space \(V\).

\begin{lemma}\label{lem:lift}
Let \(f\in C^\infty(\partial\Omega)\) and let \(F\in
C^\infty(\overline\Omega)\) be supported in \(\{d<r_0\}\), with
\(F(x-r\nu(x))=f(x)\) for \(x\in\partial\Omega\) and \(0\le r<r_0/2\). Let
\(\zeta_f=\delta^{-1/2}F\) in \(\Omega\) and \(\zeta_f=0\) outside
\(\Omega\). Then \(\zeta_f\in L^1(\R^n)\) and
\((-\Delta)^{1/2}\zeta_f\in C^\infty(\overline\Omega)\).
\end{lemma}

\begin{proof}
The function \(\zeta_f\) is integrable because \(r^{-1/2}\) is integrable
near \(0\). Let \(A=2F\nabla\delta\), then
\begin{equation}\label{eq:lift-identity}
 \zeta_f=\operatorname{div}\big(\sqrt\delta\,A\big)-\sqrt\delta\,\operatorname{div}A
\end{equation}
in the sense of distributions on \(\R^n\). The divergence produces no measure
on \(\partial\Omega\), because \(\sqrt\delta\,A\) belongs to
\(W^{1,1}(\R^n)\) (its gradient is \(O(d^{-1/2})\)) and vanishes on
\(\partial\Omega\). Since \((-\Delta)^{1/2}\) commutes with derivatives,
\begin{equation}\label{eq:lift-image}
 (-\Delta)^{1/2}\zeta_f=\sum_i\partial_i(-\Delta)^{1/2}\big(\sqrt\delta\,A_i\big)
 -(-\Delta)^{1/2}\big(\sqrt\delta\,\operatorname{div}A\big),
\end{equation}
and the second line of \eqref{eq:regularity} shows that the restriction of
the right-hand side to \(\Omega\) belongs to \(C^\infty(\overline\Omega)\).
\end{proof}

\subsection{\texorpdfstring{The Dirichlet-to-Neumann operator $\mathcal D_\Omega$}{The Dirichlet-to-Neumann operator D}}

\begin{proposition}\label{prop:large}
For every \(f\in C^\infty(\partial\Omega)\) there is a unique function
\(Q_f\), vanishing outside \(\Omega\), with \(\sqrt\delta\,Q_f\in
C^\infty(\overline\Omega)\) and satisfying \eqref{eq:large}. With
\(\zeta_f\) as in Lemma~\ref{lem:lift}, \(Q_f=\zeta_f+\omega_f\), where
\(\omega_f\in V\) solves \((-\Delta)^{1/2}\omega_f=-(-\Delta)^{1/2}\zeta_f\) in
\(\Omega\), and
\begin{equation}\label{eq:D-regular}
 \mathcal D_\Omega f=\big(\omega_f/\sqrt\delta\big)\big|_{\partial\Omega}.
\end{equation}
Moreover \(Q_f(x-r\nu(x))=f(x)\,r^{-1/2}+\mathcal D_\Omega f(x)\sqrt
r+O(r^{3/2})\) uniformly in \(x\in\partial\Omega\). Uniqueness holds in the
larger class \(\sqrt\delta\,Q\in C^2(\overline\Omega)\); in particular
\(Q_f\) does not depend on the choice of \(F\).
\end{proposition}

\begin{proof}
\emph{Existence.} Let \(g_f=(-\Delta)^{1/2}\zeta_f\),
which belongs to \(C^\infty(\overline\Omega)\) by Lemma~\ref{lem:lift}. By
\eqref{eq:L2-by-energy}, \(v\mapsto\int_\Omega g_fv\) is a bounded linear form
on \(V\), and the Lax--Milgram theorem gives a unique \(\omega_f\in V\) with
\(a(\omega_f,v)=-\int_\Omega g_fv\) for every \(v\in V\). Then
\(Q_f=\zeta_f+\omega_f\) vanishes outside \(\Omega\) and
\((-\Delta)^{1/2}Q_f=0\) in \(\Omega\).

\emph{Boundary behaviour.}  The first line of
\eqref{eq:regularity} gives \(\omega_f=\sqrt\delta\,h_f\) with \(h_f\in
C^\infty(\overline\Omega)\), hence
\[
 \sqrt\delta\,Q_f=F+\delta\,h_f .
\]
This shows that \(\sqrt\delta\,Q_f\in C^\infty(\overline\Omega)\) and that
its boundary value is \(f\). For \(x\in\partial\Omega\) and small \(r>0\) we
have \(\delta(x-r\nu(x))=r\) and \(F(x-r\nu(x))=f(x)\), so
\[
 \sqrt r\,Q_f(x-r\nu(x))=f(x)+r\,h_f(x-r\nu(x)) .
\]
Differentiating at \(r=0\) gives \(\mathcal D_\Omega f=h_f|_{\partial\Omega}\),
which is \eqref{eq:D-regular}; here it matters that \(F\) is constant along
normals, so that no derivative of \(F\) appears. The expansion of \(Q_f\)
follows from the same identity and the smoothness of \(h_f\).

\emph{Uniqueness.} Let \(Q_1,Q_2\) be two solutions of \eqref{eq:large} with
the same datum \(f\) and with \(\sqrt\delta\,Q_i\in C^2(\overline\Omega)\).
Set \(R=Q_1-Q_2\) and \(P=\sqrt\delta\,R\in C^2(\overline\Omega)\), so that
\(P=0\) on \(\partial\Omega\). Near  \(\partial\Omega\), where \(r=\delta\) is a
coordinate, \(P(z,r)=r\,A(z,r)\) with \(A(z,r)=\int_0^1\partial_rP(z,\theta
r)\,d\theta\), and \(A\in C^1\) because \(\partial_rP\in C^1\). Away from the
boundary \(\delta\) is bounded below. Hence \(P=\delta A\) with \(A\in
C^1(\overline\Omega)\), and \(R=\sqrt\delta\,A\) belongs to \(V\) by
Lemma~\ref{lem:energy-membership}.  Since \(R\) is \(\tfrac12\)-harmonic in
\(\Omega\), \(a(R,\phi)=0\) for \(\phi\in C_c^\infty(\Omega)\), hence for
every \(\phi\in V\) by density. Taking \(\phi=R\) gives \(R=0\). Two
different choices of \(F\) give two solutions of \eqref{eq:large}, which
therefore coincide.
\end{proof}

In particular \(\mathcal D_\Omega f\) is computed by solving one ordinary
Dirichlet problem, for \(\omega_f\), and then using \eqref{eq:D-regular}.

\subsection{Angular operators}

For an even function \(b\in L^\infty(\S^{n-1})\), that is
\(b(-\theta)=b(\theta)\), we consider the operator
\[
 \mathscr P_b v(y):=c_n\,\mathrm{P.V.}\!\int_{\R^n}
 (v(y)-v(z))\frac{b((z-y)/|z-y|)}{|z-y|^{n+1}}\,dz .
\]
For \(b\equiv1\) this is \((-\Delta)^{1/2}\). Passing to polar coordinates
around \(y\) and pairing the directions \(\theta\) and \(-\theta\), which is
possible since \(b\) is even, we obtain
\begin{equation}\label{eq:Pb-symmetrized}
 \mathscr P_bv(y)=\frac{c_n}{2}\int_{\S^{n-1}}b(\theta)\int_0^\infty
 \frac{2v(y)-v(y+\rho\theta)-v(y-\rho\theta)}{\rho^2}
 \,d\rho\,d\theta.
\end{equation}
Notice that as opposed to the original definition of $ \mathscr P_b$ this integral will always be absolutely convergent in the cases we are interested in.

The following lemma is the key estimate for the transported forcing in
Section~\ref{ssec:forcing}. It may be seen as a simplified version
of~\cite[Corollary~2.3]{ARO}. We include a proof because it is elementary,
 does not require \(b\) to be positive and uses only its
boundedness.

\begin{lemma}\label{lem:frozen-angular}
Let \(U\in C^\infty(\overline\Omega)\) and \(u=\sqrt\delta\,U\). For every
bounded even function \(b:\S^{n-1}\to\R\),
\begin{equation}\label{eq:lipschitz-angular}
 \|\mathscr P_bu\|_{L^\infty(\Omega)}\le C\|b\|_{L^\infty(\S^{n-1})},
\end{equation}
where \(C\) depends only on \(\Omega\), \(\delta\) and
\(\|U\|_{C^2(\overline\Omega)}\).
\end{lemma}

\begin{proof}

Let \(\mathsf d\) be the signed distance to \(\partial\Omega\), positive in
\(\Omega\), which is smooth on \(\{|\mathsf d|<2r_0\}\), and extend \(U\)
smoothly to a neighbourhood of \(\overline\Omega\). Let \(C_\Omega\) be a
bound for \(D^2\mathsf d\) on \(\{|\mathsf d|<2r_0\}\), and fix
\(R_0\in(0,r_0/5]\) with \(16C_\Omega R_0\le1\). Fix $y \in \Omega$. We will treat separately the cases $d(y)\ge R_0$ and $d(y)<R_0$.

We start with the easier case \(d(y)\ge R_0\). Since \(d\) is \(1\)-Lipschitz,
\(B_{R_0/2}(y)\) is contained in the compact set
\(K=\{z\in\overline\Omega:d(z)\ge R_0/2\}\), on which \(\delta\) is bounded
below by a positive constant. Hence \(M:=\sup_K|D^2u|<\infty\). For
\(0<\rho<R_0/2\), Taylor's formula applied to
\(\rho\mapsto u(y+\rho\theta)+u(y-\rho\theta)\), whose derivative vanishes at
\(\rho=0\), gives \(|2u(y)-u(y+\rho\theta)-u(y-\rho\theta)|\le M\rho^2\). For
\(\rho\ge R_0/2\) the numerator is bounded by \(4\|u\|_\infty\). Hence
\[
 |\mathscr P_bu(y)|
 \le\frac{c_n}{2}\,|\S^{n-1}|\,\|b\|_\infty
 \Big(\frac{MR_0}{2}+\frac{8\|u\|_\infty}{R_0}\Big)
 \le C\|b\|_\infty .
\]

If instead 
\(d(y)<R_0\), then \(y=x-d(y)\nu(x)\) for a unique \(x\in\partial\Omega\);
\(\mathsf d\) and \(U\) are smooth on \(B_{5R_0}(x)\), and
\(u=U\sqrt{\mathsf d_+}\) there. We first prove that 
\begin{equation}\label{eq:flat-annihilation}
 \mathscr P_b\big((\mathsf l)_+^{1/2}\big)(y)=0
 \qquad\text{whenever \(\mathsf l\) is affine and \(\mathsf l(y)>0\)},
\end{equation}

Let thus \(\mathsf l\) be an affine
function with \(\mathsf l(y)=r>0\), and let \(v=(\mathsf l)_+^{1/2}\). Fix
\(\theta\in\S^{n-1}\) and put \(\alpha=\nabla\mathsf l\cdot\theta\). Since
\(\mathsf l(y\pm\rho\theta)=r\pm\alpha\rho\), the inner integral in
\eqref{eq:Pb-symmetrized} is
\[
 I(r,\alpha):=\int_0^\infty
 \frac{2\sqrt r-\sqrt{(r+\alpha\rho)_+}-\sqrt{(r-\alpha\rho)_+}}{\rho^2}
 \,d\rho .
\]
We claim that it converges absolutely and that
\begin{equation}\label{eq:lipschitz-flat}
 I(r,\alpha)=0\qquad\text{for all }r>0,\ \alpha\in\R .
\end{equation}
After integration this would prove the claimed \eqref{eq:flat-annihilation}. The integrand is even in \(\alpha\) and vanishes for \(\alpha=0\). For
\(\alpha\neq0\), the substitution \(\rho=rs/|\alpha|\) gives
\[
 I(r,\alpha)=\frac{|\alpha|}{\sqrt r}\,I_0,\qquad
 I_0:=\int_0^\infty\frac{g(s)}{s^2}\,ds,\qquad
 g(s):=2-\sqrt{1+s}-\sqrt{(1-s)_+}.
\]
Clearly $|I_0|<\infty$ and we only need to prove that it is zero. Let
\[
 F_-(s)=\log\frac{1-\sqrt{1-s}}{1+\sqrt{1-s}}\quad(0<s<1),\qquad
 F_+(s)=\log\frac{\sqrt{1+s}-1}{\sqrt{1+s}+1}\quad(s>0).
\]
Integrating by parts it may be checked that 
\[ 
2 I_0= \int_0^1(F_-'-F_+') -\int_1^\infty F_+'.
\]
Since $(F_--F_+)(0)=0$, $F_-(1)=0$ and $F_+\to 0$ at infinity, this concludes the proof of $I_0=0$ and thus 
of \eqref{eq:lipschitz-flat}.

Recall now that 
\(y=x-r\nu(x)\) with \(0<r<R_0\), and set
\[
 \mathsf l_y(z)=-(z-x)\cdot\nu(x),\qquad
 h_y(z)=U(y)\sqrt{(\mathsf l_y(z))_+},\qquad e_y=u-h_y .
\]
Then \(\mathsf l_y\) is affine, \(\mathsf l_y(y)=r\), and
\(u(y)=U(y)\sqrt r=h_y(y)\), so \(e_y(y)=0\). By
\eqref{eq:flat-annihilation}, \(\mathscr P_bh_y(y)=0\), hence by linearity,
\(\mathscr P_bu(y)=\mathscr P_be_y(y)\) so that it is enough to bound $\mathscr P_be_y(y)$. We bound separately the
contributions to \eqref{eq:Pb-symmetrized} of \(|z-y|<r/2\), of \(r/2\le|z-y|<R_0\), and of
\(|z-y|\ge R_0\).

 Before doing so we record that since \(\mathsf d(x)=0\) and
\(\nabla\mathsf d(x)=-\nu(x)=\nabla\mathsf l_y\), Taylor's formula gives, for
\(z\in B_{5R_0}(x)\),
\begin{equation}\label{eq:tangent-plane-error}
 |\mathsf d(z)-\mathsf l_y(z)|\le C_\Omega|z-x|^2,\qquad
 |\nabla\mathsf d(z)-\nabla\mathsf l_y|\le C_\Omega|z-x|.
\end{equation}

We start with  \(z\in B_{r/2}(y)\) where both \(\mathsf d\) and \(\mathsf l_y\) are at least
\(r/2\), and \(|z-x|\le2r\). Writing
\(e_y=(U-U(y))\sqrt{\mathsf d}+U(y)\big(\sqrt{\mathsf d}-\sqrt{\mathsf
l_y}\big)\) and differentiating twice, we get
\begin{equation}\label{eq:lipschitz-local-error}
 \|D^2e_y\|_{L^\infty(B_{r/2}(y))}\le Cr^{-1/2}.
\end{equation}
 The second
difference thus satisfies
\(|2e_y(y)-e_y(y+\rho\theta)-e_y(y-\rho\theta)|\le Cr^{-1/2}\rho^2\) for
\(\rho<r/2\), and the near field contributes at most of order  \(\|b\|_\infty\sqrt
r\) to \(\mathscr P_be_y(y)\).

 Let now \(r/2\le\rho\le R_0\) and
consider \(z\) with \(\rho\le|z-y|<2\rho\), so that \(|z-x|\le4\rho\). Since
\(|U(z)-U(y)|\le C\rho\) and \(\sqrt{\mathsf d(z)_+}\le C\sqrt\rho\),
\[
 |e_y(z)|\le C\rho^{3/2}
 +C\Big|\sqrt{\mathsf d(z)_+}-\sqrt{\mathsf l_y(z)_+}\Big| .
\]
The bound \(|\sqrt{\mathsf d_+}-\sqrt{(\mathsf l_y)_+}|\le|\mathsf
d-\mathsf l_y|^{1/2}\le C\rho\) is not sufficient: it would give a
contribution of order one from each dyadic annulus, hence a logarithmic
divergence. We use instead that \(\mathsf d\) and \(\mathsf l_y\) differ by
\(O(\rho^2)\) in the normal direction. Set \(a=\mathsf l_y(z)\) and
\(\eps_\rho=16C_\Omega\rho^2\), so that \(|\mathsf d(z)-a|\le\eps_\rho\) by
\eqref{eq:tangent-plane-error}. Since \(s\mapsto\sqrt{s_+}\) is
nondecreasing, both \(\sqrt{\mathsf d(z)_+}\) and \(\sqrt{a_+}\) lie between
\(\sqrt{(a-\eps_\rho)_+}\) and \(\sqrt{(a+\eps_\rho)_+}\), so
\[
 \Big|\sqrt{\mathsf d(z)_+}-\sqrt{a_+}\Big|
 \le\sqrt{(a+\eps_\rho)_+}-\sqrt{(a-\eps_\rho)_+}.
\]
By the choice of \(R_0\), \(\eps_\rho\le\rho\). Integrating in the normal
coordinate \(a\),
\[
 \int_{-4\rho}^{4\rho}
 \left(\sqrt{(a+\eps_\rho)_+}-\sqrt{(a-\eps_\rho)_+}\right)da
 =\frac23\left((4\rho+\eps_\rho)^{3/2}
                 -(4\rho-\eps_\rho)^{3/2}\right)
 \le C\eps_\rho\sqrt\rho=C\rho^{5/2},
\]
and, since the tangential cross-sections have measure at most
\(O(\rho^{n-1})\),
\begin{equation}\label{eq:lipschitz-annulus-error}
 \int_{\rho\le|z-y|<2\rho}|e_y(z)|\,dz\le C\rho^{n+3/2}.
\end{equation}
As \(e_y(y)=0\) and \(|z-y|^{-n-1}\le\rho^{-n-1}\) on the annulus, each
annulus contributes at most \(O(\|b\|_\infty\rho^{1/2})\) to
\(\mathscr P_be_y(y)\). Summing over \(\rho=2^jr/2<R_0\), \(j\ge0\), gives a
bound \(O(\|b\|_\infty)\) independent of \(r\).

 For \(|z-y|\ge R_0\), \(|e_y(z)|\le\|u\|_\infty+
|U(y)|\sqrt{r+|z-y|}\le C(1+|z-y|^{1/2})\), which is integrable against
\(|z-y|^{-n-1}\) on \(\{|z-y|\ge R_0\}\). The tail contributes at most of the order of
\(\|b\|_\infty\).

Summing the three contributions we conclude that as claimed 
\[
|\mathscr P_bu(y)|=|\mathscr P_be_y(y)|\le C \|b\|_\infty.
\]
\end{proof}

\section{\texorpdfstring{Proof of Theorem~\ref{thm:hessian}}{Proof of Theorem 1.1}}
\label{sec:proof-hessian}

\subsection{Deformations}
\label{ssec:deformations}
In this section the normal deformation direction $\varphi\in C^\infty(\partial \Omega)$ is fixed. For the corresponding family of deformations $\Phi_t$ we set $\Omega_t=\Phi_t(\Omega)$,
$u_t=u_{\Omega_t}$ and 
\(J_t=\det D\Phi_t\).  For \(|t|\) small,
\(\Omega_t\) has \(C^\infty\) boundary, and we fix
\(\delta_t\in C^\infty(\overline{\Omega_t})\), positive in \(\Omega_t\) and
equal to the distance to \(\partial\Omega_t\) near \(\partial\Omega_t\), with
\(\delta_0=\delta\). The function
\(\delta\circ\Phi_t^{-1}\) is smooth on \(\overline{\Omega_t}\), positive in
\(\Omega_t\), and vanishes on \(\partial\Omega_t\) with non-zero gradient.
Writing it in normal coordinates of \(\partial\Omega_t\) and using Taylor's
formula with integral remainder in the normal variable, we obtain
\begin{equation}\label{eq:defining-quotient}
 \delta\circ\Phi_t^{-1}=\gamma_t\,\delta_t,\qquad
 \gamma_t\in C^\infty(\overline{\Omega_t}),\quad \gamma_t>0 .
\end{equation}
On the boundary, \(\gamma_t\) is the normal derivative of
\(\delta\circ\Phi_t^{-1}\): since \(\nabla\delta=-\nu\) on \(\partial\Omega\),
\begin{equation}\label{eq:stretch}
 \gamma_t\big(\Phi_t(x)\big)=\beta_t(x):=\big|D\Phi_t(x)^{-T}\nu(x)\big|,
 \qquad x\in\partial\Omega .
\end{equation}
Since \(D\Phi_t=\mathrm{Id}+t\,DX\) we have 
\begin{equation}\label{eq:stretch-derivative}
 \beta_0=1,\qquad
 \partial_t\beta_t\big|_{t=0}=-\nu\cdot DX\,\nu=0
 \quad\text{on }\partial\Omega .
\end{equation}

\begin{definition}\label{def:derivatives}
The \emph{material derivative} is \(w=\partial_tv_t|_{t=0}\), where
\(v_t=u_t\circ\Phi_t\). The \emph{Eulerian derivative} is
\(q=\partial_tu_t|_{t=0}\) in \(\mathcal D'(\R^n)\).
\end{definition}

By Proposition~\ref{prop:bulk} below, \(w\) exists in \(V\), and by
Proposition~\ref{prop:q}, \(q\) exists and
\begin{equation}\label{eq:q-def}
 q=w-X\cdot\nabla u .
\end{equation}
The function \(w\) vanishes like \(\sqrt d\) on \(\partial\Omega\)
(Lemma~\ref{lem:material-coefficient}), whereas \(X\cdot\nabla u\) behaves
like \(-\frac{\psi\varphi}2d^{-1/2}\). Hence, as soon as \(\varphi\not\equiv0\),
\(q\) does not belong to \(L^2\), and a fortiori not to \(V\): the Eulerian
derivative lies outside the variational framework in which \(u_t\) is
defined. This is the main difficulty in the computation of the second
variation.

\subsection{The bulk second variation}

We first pull the problem back to the fixed domain \(\Omega\) and write the
second derivative of \(T(\Omega_t)\) in terms of the material derivative.
Changing variables by \(\Phi_t\), we define on \(V\)
\begin{equation}\label{eq:pullback}
 \begin{split}
 a_t(v_1,v_2)&=\frac{c_n}{2}\iint
 \frac{(v_1(y)-v_1(z))(v_2(y)-v_2(z))\,J_t(y)J_t(z)}
 {|\Phi_t(y)-\Phi_t(z)|^{n+1}}\,dy\,dz,\\
 \ell_t(v)&=\int_\Omega v\,J_t .
 \end{split}
\end{equation}
Thus \(a_0=a\), \(\ell_0(v)=\int_\Omega v\), and \(a_t(v_1,v_2)=a(v_1\circ
\Phi_t^{-1},v_2\circ\Phi_t^{-1})\). We write \(a_0'\), \(a_0''\),
\(\ell_0'\), \(\ell_0''\) for the derivatives at \(t=0\), which exist in
\(\mathcal L(V,V^*)\) and \(V^*\) by the proof below.

\begin{proposition}\label{prop:bulk}
The map \(t\mapsto v_t=u_t\circ\Phi_t\) is of class \(C^2\) with values in
\(V\), the material derivative \(w\) is the unique solution in \(V\) of
\begin{equation}\label{eq:w}
 a(w,v)=\ell'_0(v)-a'_0(u,v)\qquad(v\in V),
\end{equation}
and
\begin{equation}\label{eq:bulk}
 \frac{d^2}{dt^2}T(\Omega_t)\Big|_{t=0}=2\ell''_0(u)-a''_0(u,u)+2a(w,w).
\end{equation}
\end{proposition}

\begin{proof}
\emph{The equation on the fixed domain.} Composition with the uniformly
bi-Lipschitz maps \(\Phi_t\) preserves \(H^{1/2}(\R^n)\), with uniform bounds,
and \(u_t\circ\Phi_t\) vanishes outside \(\Omega\); hence \(v_t\in V\) for
every \(t\). Changing variables in the weak formulation of \eqref{eq:torsion}
on \(\Omega_t\) and testing with \(v\circ\Phi_t^{-1}\), \(v\in V\), we obtain
\begin{equation}\label{eq:state-fixed}
 a_t(v_t,v)=\ell_t(v)\qquad(v\in V).
\end{equation}

\emph{Regularity of \(a_t\) and \(\ell_t\).} We isolate the dependence on
\(t\) in the multiplier
\begin{equation}\label{eq:rho}
 \begin{gathered}
 \rho_t(y,z)=J_t(y)J_t(z)
 \left(\frac{|y-z|}{|\Phi_t(y)-\Phi_t(z)|}\right)^{n+1},\\
 a_t(v_1,v_2)=\frac{c_n}{2}\iint
 \frac{(v_1(y)-v_1(z))(v_2(y)-v_2(z))}{|y-z|^{n+1}}\,\rho_t(y,z)\,dy\,dz .
 \end{gathered}
\end{equation}
Let \(\Xi_t(y,z)=(\Phi_t(y)-\Phi_t(z))/|y-z|\) for \(y\neq z\). The map
\(g\mapsto(g(y)-g(z))/|y-z|\) is linear and bounded from
\(W^{1,\infty}(\R^n;\R^n)\) to \(L^\infty(\R^{2n};\R^n)\). Since \(t\mapsto
\Phi_t-\mathrm{Id}\) is \(C^2\) with values in \(W^{1,\infty}\), \(t\mapsto
\Xi_t\) is \(C^2\) with values in \(L^\infty(\R^{2n})\). The uniform
bi-Lipschitz bounds show that \(\Xi_t\) takes values in a fixed compact subset
of \(\R^n\setminus\{0\}\), on which \(\xi\mapsto|\xi|^{-n-1}\) is smooth, so
\(t\mapsto|\Xi_t|^{-n-1}\) is also \(C^2\) into \(L^\infty\). Since \(J_t\) is a
polynomial in the entries of \(D\Phi_t\), we conclude that
\begin{equation}\label{eq:rho-C2}
 \rho_t\in C^2\big((-t_0,t_0);L^\infty(\R^{2n})\big).
\end{equation}
Note that this is differentiability with respect to the \(L^\infty\) norm,
not merely a pointwise bound on \(\partial_t^k\rho_t\). For \(m\in
L^\infty(\R^{2n})\) let \(\mathsf B_m\) be the bilinear form defined as in
\eqref{eq:rho} with \(\rho_t\) replaced by \(m\). The Cauchy--Schwarz
inequality gives
\begin{equation}\label{eq:Bm}
 |\mathsf B_m(v_1,v_2)|\le\|m\|_{L^\infty}\|v_1\|_V\|v_2\|_V ,
\end{equation}
so \(m\mapsto\mathsf B_m\) is a bounded linear map from
\(L^\infty(\R^{2n})\) into \(\mathcal L(V,V^*)\). Since \(a_t=\mathsf
B_{\rho_t}\), \eqref{eq:rho-C2} and \eqref{eq:Bm} give
\(a_t\in C^2((-t_0,t_0);\mathcal L(V,V^*))\), with
\(\partial_t^ka_t=\mathsf B_{\partial_t^k\rho_t}\). The uniform bounds
\(0<c\le\rho_t\le C\) give for $v\in V$,
\begin{equation}\label{eq:unif-coercive}
 c\,a(v,v)\le a_t(v,v)\le C\,a(v,v).
\end{equation}
By \eqref{eq:L2-by-energy}, \(|\int_\Omega vJ|\le C\|J\|_{L^\infty}\|v\|_V\);
since \(J_t\in C^2_tL^\infty_x\), we get \(\ell_t\in C^2((-t_0,t_0);V^*)\).

\emph{Differentiability of \(v_t\).} Let \(A_t\in\mathcal L(V,V^*)\) be
defined by \(\langle A_tv_1,v_2\rangle=a_t(v_1,v_2)\), so that
\eqref{eq:state-fixed} reads \(A_tv_t=\ell_t\). By \eqref{eq:unif-coercive}
and the Lax--Milgram theorem, each \(A_t\) is invertible with
\(\|A_t^{-1}\|\le c^{-1}\). Inversion is smooth on the open set of invertible
operators, so \(t\mapsto A_t^{-1}\) is \(C^2\) into \(\mathcal L(V^*,V)\), and
\(v_t=A_t^{-1}\ell_t\in C^2((-t_0,t_0);V)\). Differentiating
\eqref{eq:state-fixed} at \(t=0\) gives \(a'_0(u,v)+a(w,v)=\ell'_0(v)\), that
is \eqref{eq:w}, which determines \(w\) by coercivity.

\emph{The second variation.} Let \(F_t(v)=2\ell_t(v)-a_t(v,v)\). By
\eqref{eq:state-fixed}, \(F_t(v_t)=\ell_t(v_t)=\int_{\Omega_t}u_t=T(\Omega_t)\).
Set \(\mathcal G(t)=F_t(u)=2\ell_t(u)-a_t(u,u)\) and \(h=v_t-u\). Using
\eqref{eq:state-fixed},
\[
 F_t(v_t)-F_t(u)=2\ell_t(h)-2a_t(u,h)-a_t(h,h)
 =2a_t(v_t,h)-2a_t(u,h)-a_t(h,h)=a_t(h,h),
\]
which gives the exact identity
\begin{equation}\label{eq:bulk-exact}
 T(\Omega_t)=\mathcal G(t)+a_t(v_t-u,v_t-u).
\end{equation}
The last term is \(C^2\) in \(t\), and since \(v_t-u=tw+o(t)\) in \(V\) and
\(a_t\to a\) in operator norm, it equals \(t^2a(w,w)+o(t^2)\); its second
derivative at \(t=0\) is \(2a(w,w)\). Since \(u\) is fixed,
\(\mathcal G''(0)=2\ell''_0(u)-a''_0(u,u)\), and \eqref{eq:bulk} follows.
\end{proof}

\subsection{The Eulerian derivative}

We now identify the Eulerian derivative with the large solution of
Proposition~\ref{prop:large} with datum \(\psi\varphi/2\). This is the step
where the singular behaviour of \(q\) is handled.

\begin{proposition}\label{prop:q}
The Eulerian derivative is the integrable function
\begin{equation}\label{eq:q}
 q=w-X\cdot\nabla u=Q_{\psi\varphi/2},
\end{equation}
so that, for \(x\in\partial\Omega\),
\begin{equation}\label{eq:q-expansion}
 q(x-r\nu(x))=\frac{\psi\varphi}{2\sqrt r}+B_\varphi\sqrt r+O(r^{3/2}),
 \qquad B_\varphi=\tfrac12\mathcal D_\Omega(\psi\varphi).
\end{equation}
\end{proposition}

\begin{proof}
Since \(|\nabla u|\le C\delta^{-1/2}\) near \(\partial\Omega\), the gradient
is integrable. Moreover, \(u\) vanishes on \(\partial\Omega\), so its zero
extension creates no jump and \(u=\sqrt\delta\,U\in W^{1,1}(\R^n)\).
 Let \(\chi\in
C_c^\infty(\R^n)\). Changing variables,
\[
 \int_{\R^n}u_t\chi=\int_\Omega v_t\,(\chi\circ\Phi_t)\,J_t ,
\]
and since \(v_t\in C^1_tV\), this is differentiable at \(t=0\) with
derivative
\[
 \langle q,\chi\rangle=\int_\Omega w\chi+\int_\Omega u\,X\cdot\nabla\chi
 +\int_\Omega u\chi\operatorname{div}X .
\]
Since \(uX\chi\in W^{1,1}(\R^n)\) has compact support,
\(\int_{\R^n}\operatorname{div}(uX\chi)=0\), and the product rule turns the
last two terms into \(-\int\chi\,X\cdot\nabla u\). Hence
\[q=w-X\cdot\nabla u\in L^1(\R^n).\]

Let now \(\chi\in
C_c^\infty(\Omega)\). For \(|t|\) small the support of \(\chi\) is contained
in \(\Omega_t\), so \(\int u_t(-\Delta)^{1/2}\chi=\int\chi\), which does not
depend on \(t\). All the functions \(u_t\), and hence \(q\), are supported in a
fixed compact set \(K\). Let \(\zeta\in C_c^\infty(\R^n)\) with \(\zeta=1\) near
\(K\). Then \(\int u_t\,\zeta(-\Delta)^{1/2}\chi=\int\chi\) for all small
\(|t|\), and the formula for \(\langle q,\cdot\rangle\) obtained above gives
\(\langle q,(-\Delta)^{1/2}\chi\rangle=\langle q,\zeta(-\Delta)^{1/2}\chi\rangle=0\),
that is, \(q\) is \(\tfrac12\)-harmonic in \(\Omega\) in the sense of distributions.

Let \(f=\psi\varphi/2\), and
let \(F\) and \(\zeta_f=\delta^{-1/2}F\) be as in Lemma~\ref{lem:lift}. From
\(u=\sqrt\delta\,U\),
\begin{equation}\label{eq:energy-cancellation}
 X\cdot\nabla u+\zeta_f
 =\delta^{-1/2}G+\sqrt\delta\,X\cdot\nabla U,
 \qquad G:=\tfrac12U\,X\cdot\nabla\delta+F .
\end{equation}
The function \(G\) is smooth up to \(\partial\Omega\). On \(\partial\Omega\)
we have \(U=\psi\), \(\nabla\delta=-\nu\), hence \(X\cdot\nabla
\delta=-\varphi\), and \(F=\psi\varphi/2\), so that
\[
 G|_{\partial\Omega}=-\tfrac12\psi\varphi+\tfrac12\psi\varphi=0 .
\]
This is the point of the choice \(f=\psi\varphi/2\): neither
\(X\cdot\nabla u\) nor \(\zeta_f\) belongs to \(V\), but their \(d^{-1/2}\)
singularities cancel. As in the proof of uniqueness in
Proposition~\ref{prop:large}, \(G=\delta A\) with \(A\in
C^\infty(\overline\Omega)\), and
\begin{equation}\label{Xzetaf}
 X\cdot\nabla u+\zeta_f=\sqrt\delta\,\big(A+X\cdot\nabla U\big)\in V
\end{equation}
by Lemma~\ref{lem:energy-membership}.

By Proposition~\ref{prop:large},
\(Q_f=\zeta_f+\omega_f\) with \(\omega_f\in V\). Hence
\begin{equation}\label{eq:compare}
 q-Q_f=w-\big(X\cdot\nabla u+\zeta_f\big)-\omega_f\ \in V ,
\end{equation}
each of the three terms being in \(V\), by Proposition~\ref{prop:bulk},
\eqref{Xzetaf} and Proposition~\ref{prop:large}. Both \(q\) and \(Q_f\) are \(\tfrac12\)-harmonic in
\(\Omega\), by the argument above and \eqref{eq:large}, so \(a(q-Q_f,\phi)=0\) for
\(\phi\in C_c^\infty(\Omega)\), and by density for every \(\phi\in V\).
Taking \(\phi=q-Q_f\) gives \(q=Q_f\). The expansion \eqref{eq:q-expansion}
follows from Proposition~\ref{prop:large} and the linearity of
\(\mathcal D_\Omega\).
\end{proof}

\subsection{The material derivative}

\begin{lemma}\label{lem:material-coefficient}
We have \(w=\sqrt\delta\,W\) with \(W\in C^\infty(\overline\Omega)\), and
\(f_w:=(-\Delta)^{1/2}w\in C^\infty(\overline\Omega)\).
Moreover \(\psi_w:=W|_{\partial\Omega}\) is given by
\begin{equation}\label{eq:material-coefficient}
 \psi_w=B_\varphi-\tfrac32\kappa\varphi.
\end{equation}
\end{lemma}

Proposition~\ref{prop:bulk} only gives \(w\in V\), which does not imply that
\(w/\sqrt d\) has a boundary trace. The lemma provides the structure of \(w\)
needed to apply the Pohozaev identity of Lemma~\ref{lem:bilinear-pohozaev} to
the pair \((u,w)\).

\begin{proof}
Let \(f=\psi\varphi/2\) so that  by Proposition~\ref{prop:large},
\(\omega_f=\sqrt\delta\,h_f\) with \(h_f\in C^\infty(\overline\Omega)\).
By \eqref{eq:q},
\[
 w=q+X\cdot\nabla u=\big(X\cdot\nabla u+\zeta_f\big)+\omega_f
 \stackrel{\eqref{Xzetaf}}{=}\sqrt\delta\,\big(A+X\cdot\nabla U+h_f\big),
\]
so \(w=\sqrt\delta\,W\) with \(W\in C^\infty(\overline\Omega)\). The second
line of \eqref{eq:regularity} then gives \(f_w\in C^\infty(\overline\Omega)\).

To compute \(\psi_w\), we expand \(w=q+X\cdot\nabla u\) along the normal ray
\(x-r\nu(x)\), \(x\in\partial\Omega\). There \(\nabla\delta=-\nu\) and
\(U(x-r\nu)=\psi+\kappa r+O(r^2)\), so
\[
 X\cdot\nabla u=-\frac{U}{2\sqrt r}\,X\cdot\nu+\sqrt r\,X\cdot\nabla U .
\]
Recalling that $\nu\cdot DX\nu=0$ and thus \(X(x-r\nu)\cdot\nu=\varphi+O(r^2)\), the first term
equals
\(-\frac{\psi\varphi}{2\sqrt r}-\frac{1}2\sqrt r \kappa\varphi+O(r^{3/2})\). On \(\partial\Omega\),
\(\nabla U=\nabla_{\partial\Omega}\psi-\kappa\nu\), hence \(X\cdot\nabla
U=-\kappa\varphi\), and the second term
equals \(-\sqrt r\kappa\varphi
+O(r^{3/2})\). Adding \eqref{eq:q-expansion}, the terms of order
\(r^{-1/2}\) cancel and
\[
 w(x-r\nu(x))=\sqrt r\,\Big[B_\varphi-\tfrac32\kappa\varphi
 \Big]
 +O(r^{3/2}),
\]
which gives \eqref{eq:material-coefficient}. 
\end{proof}

\subsection{The bilinear Pohozaev identity}
\label{ssec:pohozaev}

By Proposition~\ref{prop:bulk}, \(\frac{d^2}{dt^2}T(\Omega_t)|_{t=0}
=\mathcal G''(0)+2a(w,w)\). We now use a Pohozaev identity to write
\(a(w,w)\) as a volume integral plus a boundary integral. We recall the following fractional Pohozaev identity from \cite[Theorem~1.3, (1.10)]{DFWPohozaev}.

\begin{lemma}\label{lem:bilinear-pohozaev}
For \(i=1,2\) let \(v_i=\sqrt\delta\,V_i\) with \(V_i\in
C^\infty(\overline\Omega)\), and set
\(f_i=(-\Delta)^{1/2}v_i\) and
\(\psi_i=V_i|_{\partial\Omega}\). For every Lipschitz vector field \(X\) on
\(\R^n\), with \(\varphi=X\cdot\nu\),
\begin{equation}\label{eq:bilinear-pohozaev}
 a'_0[X](v_1,v_2)=-\int_\Omega
 \big\{f_1\,X\cdot\nabla v_2+f_2\,X\cdot\nabla v_1\big\}\,dy
 -c_*\int_{\partial\Omega}\psi_1\psi_2\,\varphi\dHaus .
\end{equation}
\end{lemma}

 For \(X(y)=y\) and \(v_1=v_2\), \eqref{eq:bilinear-pohozaev}
reduces to the fractional Pohozaev identity of Ros-Oton and
Serra~\cite{RS}, since then \(a'_0[X]=(n-1)\,a\).

\begin{remark}\label{rem:first-variation}
The case \(v_1=v_2=u\) of \eqref{eq:bilinear-pohozaev} gives a proof of the
first variation formula \eqref{eq:hadamard}. Indeed, \eqref{eq:bulk-exact}
shows that \(\frac{d}{dt}T(\Omega_t)|_{t=0}=\mathcal
G'(0)=2\ell'_0(u)-a'_0(u,u)\). Since \(u\in W^{1,1}(\R^n)\) vanishes on
\(\partial\Omega\), \(\ell'_0(u)=\int_\Omega u\operatorname{div}X
=-\int_\Omega X\cdot\nabla u\), while \eqref{eq:bilinear-pohozaev} with
\(f_1=f_2=1\) gives \(a'_0(u,u)=-2\int_\Omega X\cdot\nabla
u-c_*\int_{\partial\Omega}\psi^2\varphi\). The volume integrals cancel, and
\eqref{eq:hadamard} follows.
\end{remark}

\begin{corollary}\label{cor:poho-energy}
With the notation previously introduced we have
\begin{equation}\label{eq:poho-energy}
 a(w,w)=\int_\Omega f_w\,X\cdot\nabla u\,dy
        +c_*\int_{\partial\Omega}\psi\psi_w\varphi\dHaus .
\end{equation}
\end{corollary}

\begin{proof}
By Lemma~\ref{lem:material-coefficient}, \(w=\sqrt\delta\,W\) with \(W\)
smooth and \(f_w\) smooth, so Lemma~\ref{lem:bilinear-pohozaev} applies to
\(v_1=u\), \(v_2=w\) and gives
\[
 a'_0(u,w)=-\int_\Omega X\cdot\nabla w
 -\int_\Omega f_w\,X\cdot\nabla u
 -c_*\int_{\partial\Omega}\psi\psi_w\varphi .
\]
The zero extension of \(w\) belongs to \(W^{1,1}(\R^n)\) and vanishes on
\(\partial\Omega\), so \(\ell'_0(w)=\int_\Omega w\operatorname{div}X=
-\int_\Omega X\cdot\nabla w\). Inserting both identities in \eqref{eq:w}
with \(v=w\), that is \(a(w,w)=\ell'_0(w)-a'_0(u,w)\), the two integrals of
\(X\cdot\nabla w\) cancel and \eqref{eq:poho-energy} follows.
\end{proof}

From \eqref{eq:poho-energy} we see that  \(2a(w,w)\) contains the volume integral \(2\int_\Omega
f_wX\cdot\nabla u\). We show in the next subsection that \(\mathcal G''(0)\)
contains the same integral with the opposite sign.

\subsection{The transported forcing}
\label{ssec:forcing}

 Let
\[
 \tilde u_t=u\circ\Phi_t^{-1},\qquad
 \tilde f_t=(-\Delta)^{1/2}\tilde u_t.
\]
The function \(\tilde u_t\) is the torsion function of \(\Omega\) transported
to \(\Omega_t\) by \(\Phi_t\); it is not the torsion function of
\(\Omega_t\). By \eqref{eq:defining-quotient},
\begin{equation}\label{eq:transported-amplitude}
 \tilde u_t=\sqrt{\delta\circ\Phi_t^{-1}}\,(U\circ\Phi_t^{-1})
 =\sqrt{\delta_t}\,A_t,\qquad
 A_t:=\sqrt{\gamma_t}\,(U\circ\Phi_t^{-1})\in C^\infty(\overline{\Omega_t}),
\end{equation}
and the second line of
\eqref{eq:regularity}, applied on the domain \(\Omega_t\), gives \(\tilde
f_t\in C^\infty(\overline{\Omega_t})\) for each \(t\). Pulling back by
\(\Phi_t\), we get \(\tilde f_t\circ\Phi_t=L_tu\) in \(\Omega\), where
\begin{equation}\label{eq:Lt}
 L_tv(y)=c_n\,\mathrm{P.V.}\!\!\int_{\R^n}\frac{v(y)-v(z)}{|y-z|^{n+1}}\,
 \varrho_t(y,z)\,dz,
 \qquad
 \varrho_t(y,z)=J_t(z)\left(\frac{|y-z|}{|\Phi_ty-\Phi_tz|}\right)^{n+1}.
\end{equation}
Here the principal value is taken over the balls \(B_\eps(y)\), whereas the
change of variables produces the excised sets
\(S_\eps=\{z:|\Phi_tz-\Phi_ty|<\eps\}\). The two limits coincide for
\(v\) smooth near the fixed point \(y\in\Omega\), as is the case for \(u\).
Indeed, writing \(h=z-y\), the leading singular term of the integrand is
\(-J_t(y)\nabla v(y)\cdot h/|D\Phi_t(y)h|^{n+1}\), which is odd; the
remainder is \(O(|h|^{1-n})\) and is locally integrable. The odd term
integrates to zero between a centred ball and the centred ellipsoid
\(\{|D\Phi_t(y)h|<\eps\}\). By Taylor's formula, replacing this ellipsoid
by \(S_\eps-y\) changes its radial boundary by \(O(\eps^2)\), producing
an error \(O(\eps)\).

In \eqref{eq:Lt}, the domain and the function are fixed, and only the
kernel depends on \(t\): \((L_t)\) is a smooth family of operators of order
one. Moreover \(L_0=(-\Delta)^{1/2}\), so that \(L_0u=1\) in \(\Omega\), and
\(\rho_t(y,z)=J_t(y)\varrho_t(y,z)\), where \(\rho_t\) is the multiplier of
\eqref{eq:rho}.

\begin{lemma}\label{lem:G-prime}
Let \(X_t=\dot\Phi_t\circ\Phi_t^{-1}\),
\(\hat X_t=(D\Phi_t)^{-1}\dot\Phi_t\), let \(\nu_t\) be the outer normal to
\(\Omega_t\), and let
\begin{equation}\label{eq:Bt-def}
 \mathcal B(t)=c_*\int_{\partial\Omega_t}\tilde\psi_t^{\,2}\,(X_t\cdot\nu_t)\dHaus,
 \qquad
 \tilde\psi_t=\big(\tilde u_t/\sqrt{\delta_t}\big)\big|_{\partial\Omega_t}.
\end{equation}
Then, for all small \(t\),
\begin{equation}\label{eq:poho-Gprime}
 \mathcal G'(t)=2\int_\Omega\big(L_tu-1\big)\,\hat X_t\cdot\nabla u\;J_t\,dy+\mathcal B(t).
\end{equation}
\end{lemma}

\begin{proof}
We have \(\mathcal G(t)=2\int_{\Omega_t}\tilde u_t-a(\tilde u_t,\tilde u_t)\).
Since \(\partial_t\tilde u_t=-X_t\cdot\nabla\tilde u_t\) and \(\tilde u_t\)
vanishes on \(\partial\Omega_t\),
\(\frac{d}{dt}\int_{\Omega_t}\tilde u_t=-\int_{\Omega_t}X_t\cdot\nabla\tilde
u_t\). For the second term, let \(\Psi_h=\Phi_{t+h}\circ\Phi_t^{-1}\), so
that \(\Psi_0=\mathrm{Id}\), \(\partial_h\Psi_h|_{h=0}=X_t\), and
\(a(\tilde u_{t+h},\tilde u_{t+h})=a(\tilde u_t\circ\Psi_h^{-1},\tilde
u_t\circ\Psi_h^{-1})\). Hence \(\frac{d}{dt}a(\tilde u_t,\tilde u_t)\) is the
first-order form \(a'_0[X_t]\), computed on the domain \(\Omega_t\), applied
to \((\tilde u_t,\tilde u_t)\). By Lemma~\ref{lem:bilinear-pohozaev} on
\(\Omega_t\), with \(v_1=v_2=\tilde u_t=\sqrt{\delta_t}\,A_t\) (see
\eqref{eq:transported-amplitude}) and regularised distance \(\delta_t\),
\[
 \frac{d}{dt}a(\tilde u_t,\tilde u_t)
 =-2\int_{\Omega_t}\tilde f_t\,X_t\cdot\nabla\tilde u_t-\mathcal B(t).
\]
 Combining the two derivatives,
\[
 \mathcal G'(t)=2\int_{\Omega_t}(\tilde f_t-1)\,X_t\cdot\nabla\tilde u_t
 +\mathcal B(t).
\]
Changing variables by \(\Phi_t\) and using \(\tilde f_t\circ\Phi_t=L_tu\)
and \((X_t\cdot\nabla\tilde u_t)\circ\Phi_t=\hat X_t\cdot\nabla u\), we
obtain \eqref{eq:poho-Gprime}.
\end{proof}

Since \(L_0u-1=0\), the derivatives of \(\hat X_t\) and \(J_t\) do not
contribute to \(\mathcal G''(0)\): we only need the limit of
\((L_tu-1)/t\), tested against \(X\cdot\nabla u\). The boundary term is
explicit.

\begin{lemma}\label{lem:Bprime}
We have \(\tilde\psi_t\circ\Phi_t=\sqrt{\beta_t}\,\psi\) on \(\partial\Omega\), and
\begin{equation}\label{eq:poho-Bprime}
 \mathcal B'(0)=c_*\int_{\partial\Omega}H\psi^2\varphi^2\dHaus.
\end{equation}
\end{lemma}

\begin{proof}
By \eqref{eq:transported-amplitude} and \eqref{eq:stretch},
\(\tilde\psi_t(x_t)=A_t(x_t)=\sqrt{\gamma_t(x_t)}\,U(x)=\sqrt{\beta_t(x)}\,\psi(x)\).
On \(\partial\Omega\), \(\Phi_t(x)=x_t\) and \(X_t(x_t)=\dot\Phi_t(x)
=\varphi(x)\nu(x)\). Pulling back by \(x\mapsto x_t\),
\[
 \mathcal B(t)=c_*\int_{\partial\Omega}\psi^2\,\beta_t\,\varphi\,
 \nu\cdot\nu_t(x_t)\,J^{\partial}_t\dHaus,
\]
where \(J^\partial_t\) is the surface Jacobian of \(x\mapsto x_t\). By
\eqref{eq:stretch-derivative}, \(\beta_0=1\) and \(\partial_t\beta_t|_{t=0}=0\).
Since \(\nu_t(x_t)\) is a unit vector, \(\nu\cdot\partial_t\nu_t(x_t)|_{t=0}=0\),
and \(\partial_tJ^\partial_t|_{t=0}=\operatorname{div}_{\partial\Omega}
(\varphi\nu)=H\varphi\). This gives \eqref{eq:poho-Bprime}.
\end{proof}

We now identify the limit of \((L_tu-1)/t\). Recall from
Lemma~\ref{lem:material-coefficient} that
\(f_w=((-\Delta)^{1/2}w)|_\Omega\in C^\infty(\overline\Omega)\).

\begin{lemma}\label{lem:distributional-limit}
\((\tilde f_t-1)/t\to-f_w\) in \(\mathcal D'(\Omega)\) as \(t\to0\).
\end{lemma}

\begin{proof}
Since \(u\in W^{1,1}(\R^n)\) and \(\Phi_t=\mathrm{Id}+tX\) with \(X\) smooth and
compactly supported, \((\tilde u_t-u)/t\to-X\cdot\nabla u\) in \(L^1(\R^n)\).
Let \(\chi\in C_c^\infty(\Omega)\). For \(|t|\) small the support of \(\chi\)
is contained in \(\Omega_t\), so that
\(\int\tilde u_t(-\Delta)^{1/2}\chi=\int\tilde f_t\chi\) and \(\int
u(-\Delta)^{1/2}\chi=\int\chi\). Since \((-\Delta)^{1/2}\chi\) is bounded,
\[
 \int\frac{\tilde f_t-1}{t}\,\chi\longrightarrow
 -\int X\cdot\nabla u\,(-\Delta)^{1/2}\chi .
\]
By \eqref{eq:q}, \(-X\cdot\nabla u=q-w\). Since \(q\) is
\(\tfrac12\)-harmonic in \(\Omega\) (Proposition~\ref{prop:q}) and
\((-\Delta)^{1/2}w=f_w\) in \(\Omega\), the right-hand side equals
\(-\int f_w\chi\).
\end{proof}

This convergence is not enough to pass to the limit in
\eqref{eq:poho-Gprime}, because \(X\cdot\nabla u\) behaves like \(d^{-1/2}\):
it is integrable, but it is not a test function. We need the following
estimate.

\begin{lemma}\label{lem:lipschitz}
Let \(U\in C^\infty(\overline\Omega)\), \(u=\sqrt\delta\,U\), and let
\(L_t\) be as in \eqref{eq:Lt}. There are \(t_0>0\) and \(C\), depending
only on \(\Omega\), \(\delta\), \(\|U\|_{C^2(\overline\Omega)}\) and
\(\sup_{|t|\le t_0}|t|^{-1}\|\Phi_t-\mathrm{Id}\|_{C^2}\), such that
\begin{equation}\label{eq:lipschitz}
 \big\|L_tu-L_0u\big\|_{L^\infty(\Omega)}\le C|t|,\qquad |t|\le t_0 .
\end{equation}
\end{lemma}

The difficulty is the boundary behaviour of \(u\): \(|\nabla u|\les
d^{-1/2}\) and \(|D^2u|\les d^{-3/2}\). If one splits the principal value in \eqref{eq:Lt} at
\(|y-z|=d(y)/2\), freezes the kernel at \(y\), and expands
\(u(y)-u(z)\) to second order, the linear term cancels by parity of the
frozen kernel, but the quadratic term only gives
\(|L_tu-L_0u|\les|t|\,d^{-1/2}\). The proof uses instead the exact
cancellation of Lemma~\ref{lem:frozen-angular}.

\begin{proof}
We freeze the moving kernel at the evaluation point. Let
\[
 K_t(y,z)=\frac{J_t(z)}{|\Phi_tz-\Phi_ty|^{n+1}},\qquad
 K_t^{\rm fr}(y,z)=\frac{J_t(y)}{|D\Phi_t(y)(z-y)|^{n+1}},\qquad
 E_t=K_t-K_t^{\rm fr},
\]
so that
 \begin{multline*}(L_t-L_0)u(y)=c_n\,\mathrm{P.V.}\!\int(u(y)-u(z))\frac{b_{t,y}((z-y)/|z-y|)}{|z-y|^{n+1}}\,dz\\+
       c_n\mathrm{P.V.}\!\int_{\R^n}(u(y)-u(z))E_t(y,z)\,dz ,\end{multline*}
 where 
\[
 b_{t,y}(\theta)=\frac{J_t(y)}{|D\Phi_t(y)\theta|^{n+1}}-1,
 \qquad\theta\in\S^{n-1},
\]
which is even and satisfies \(\sup_y\|b_{t,y}\|_\infty\le C|t|\). We estimate each of the terms separately. By Lemma~\ref{lem:frozen-angular}, applied at each fixed \(y\) with
\(b=b_{t,y}\), \[|\mathscr P_{b_{t,y}}u(y)|\le C\|b_{t,y}\|_\infty\le C|t|.\] 
We now turn to the term involving $E_t$.
 Since
\(\Phi_0=\mathrm{Id}\), we have \[\|D\Phi_t-\mathrm{Id}\|_\infty
+\|D^2\Phi_t\|_\infty+\|J_t-1\|_\infty+\|\nabla J_t\|_\infty\le C|t|.\]
Using \(|\Phi_tz-\Phi_ty-D\Phi_t(y)(z-y)|\le C|t||z-y|^2\),
\(|J_t(z)-J_t(y)|\le C|t||z-y|\), and the fact that both \(K_t\) and
\(K_t^{\rm fr}\) differ from \(|z-y|^{-n-1}\) by at most
\(|t||z-y|^{-n-1}\), we get
\begin{equation}\label{eq:lipschitz-kernel-error}
 |E_t(y,z)|\le C|t|
 \begin{cases}
 |z-y|^{-n},&0<|z-y|\le1,\\
 |z-y|^{-n-1},&|z-y|>1.
 \end{cases}
\end{equation}
The error kernel is one order less singular than \(|z-y|^{-n-1}\), and the
global H\"older continuity of \(u\) is enough to control it. Indeed, the zero
extension of \(\delta\) is Lipschitz, so the zero extension of \(u\) is
\(C^{1/2}\) on \(\R^n\), and
\[
 \sup_{y\in\Omega}\left|\int_{\R^n}(u(y)-u(z))E_t(y,z)\,dz\right|
 \le C|t|\left([u]_{C^{1/2}(\R^n)}\int_0^1\rho^{-1/2}\,d\rho
                 +\|u\|_\infty\int_1^\infty\rho^{-2}\,d\rho\right)
 \le C|t|.
\]
This proves \eqref{eq:lipschitz}.
\end{proof}

\begin{proposition}\label{prop:force-limit}
We have
\begin{equation}\label{eq:poho-geometric-half}
 \mathcal G''(0)=-2\int_\Omega f_w\,X\cdot\nabla u\,dy
 +c_*\int_{\partial\Omega}H\psi^2\varphi^2\dHaus .
\end{equation}
\end{proposition}

\begin{proof}
Let \(g_t=(L_tu-1)/t\). By Lemma~\ref{lem:lipschitz} and \(L_0u=1\),
\(\|g_t\|_{L^\infty(\Omega)}\le C\). Let \(\chi\in C_c^\infty(\Omega)\) and
\(\chi_t=(\chi\circ\Phi_t^{-1})\,|\det D\Phi_t^{-1}|\). Since
\(g_t=((\tilde f_t-1)/t)\circ\Phi_t\) and
\(\|(\tilde f_t-1)/t\|_{L^\infty(\Omega_t)}=\|g_t\|_{L^\infty(\Omega)}\),
\[
 \int_\Omega g_t\chi=\int\frac{\tilde f_t-1}{t}\,\chi_t
 =\int\frac{\tilde f_t-1}{t}\,\chi+O\big(\|\chi_t-\chi\|_{L^1}\big),
\]
and Lemma~\ref{lem:distributional-limit} gives \(\int g_t\chi\to-\int
f_w\chi\). Since \((g_t)\) is bounded in \(L^\infty(\Omega)\) and
\(C_c^\infty(\Omega)\) is dense in \(L^1(\Omega)\),
\[
 g_t\overset{*}{\rightharpoonup}-f_w\qquad\text{in }L^\infty(\Omega).
\]
On the other hand \(\hat X_t\to X\) and \(J_t\to1\) uniformly, and
\(|\nabla u|\le C\delta^{-1/2}\in L^1(\Omega)\), so \(\hat
X_t\cdot\nabla u\,J_t\to X\cdot\nabla u\) strongly in \(L^1(\Omega)\).
Hence
\begin{equation}\label{eq:formal-Lt-tested}
 \int_\Omega g_t\,\hat X_t\cdot\nabla u\,J_t
 \longrightarrow-\int_\Omega f_w\,X\cdot\nabla u .
\end{equation}
By Proposition~\ref{prop:bulk}, \(\mathcal G\) is of class \(C^2\), and by
Lemma~\ref{lem:G-prime}, \(\mathcal G'(0)=\mathcal B(0)\). Therefore
\[
 \mathcal G''(0)=\lim_{t\to0}\frac{\mathcal G'(t)-\mathcal G'(0)}{t}
 =2\lim_{t\to0}\int_\Omega g_t\,\hat X_t\cdot\nabla u\,J_t+\mathcal B'(0),
\]
so that  \eqref{eq:formal-Lt-tested} and \eqref{eq:poho-Bprime} give
\eqref{eq:poho-geometric-half}.
\end{proof}

\subsection{Conclusion of the proof}
\label{ssec:assembly}

\begin{proof}[Proof of Theorem~\ref{thm:hessian}]
 By
Proposition~\ref{prop:bulk}, \(t\mapsto T(\Omega_t)\) is of class
\(C^2\). Adding \eqref{eq:poho-geometric-half} and twice
\eqref{eq:poho-energy}, the volume integrals cancel, and
\begin{equation}\label{eq:diagonal}
 \frac{d^2}{dt^2}T(\Omega_t)\Big|_{t=0}
 =c_*\int_{\partial\Omega}\big\{H\psi^2\varphi^2+2\psi\varphi\,\psi_w\big\}\dHaus .
\end{equation}
For the normal extension, \(\psi_w=\frac12\mathcal D_\Omega(\psi\varphi)
-\frac32\kappa\varphi\) by Lemma~\ref{lem:material-coefficient} and
\eqref{eq:q-expansion}, which gives \eqref{eq:hessian}.

\end{proof}

\section{The case of the ball}
\label{sec:ball}

In this section we prove Theorem~\ref{thm:ball-coercivity}. We first compute explicitly the first and second shape derivatives at the ball. Let \(B=B_1\), so that \(\nu(x)=x\) and \(H=n-1\). We use the notation
\(\mathcal P,\Lambda,\varphi_\ell\) introduced in Section~\ref{sec:intro}.
The torsion function of \(B\) is given by~\cite{Getoor,Dyda}
\begin{equation}\label{eq:ball-state}
 u(y)=A_n\sqrt{1-|y|^2},\qquad \text{with\ }
 A_n=\frac{\Gamma(n/2)}{\sqrt\pi\,\Gamma((n+1)/2)},
 \quad \psi=\sqrt2\,A_n,\quad \kappa=-\frac\psi4.
\end{equation}
The boundary coefficients follow by expanding
\(\sqrt{1-|x-r\nu(x)|^2}=\sqrt{2r}(1-r/4+O(r^2))\). In particular
\(\lambda_*=c_*\psi^2=\frac\pi2A_n^2\).

\begin{lemma}\label{lem:ball-large}
For \(f\in C^\infty(\S^{n-1})\),
\begin{equation}\label{eq:ball-large}
 Q_f(y)=\frac{\sqrt2}{\sqrt{1-|y|^2}}\,\mathcal Pf(y)\quad(y\in B),
 \qquad \mathcal D_Bf=\frac f4-\Lambda f.
\end{equation}
\end{lemma}

\begin{proof}
The Martin kernel of \(B\) for \((-\Delta)^{1/2}\) is proportional to
\((1-|y|^2)^{1/2}|y-z|^{-n}\), \(z\in\S^{n-1}\)
\cite[(1.8)]{AJS}; see also~\cite{Aba}. Comparing it with the classical
Poisson kernel shows that the displayed function, extended by zero, is
\(\tfrac12\)-harmonic in \(B\). It has smooth weighted trace \(f\), so
Proposition~\ref{prop:large} identifies it with \(Q_f\). Moreover,
\[
 \sqrt r\,Q_f(x-r\nu)
 =(1-r/2)^{-1/2}\mathcal Pf(x-r\nu)
 =f+r\big(f/4-\Lambda f\big)+O(r^2),
\]
which gives the formula for \(\mathcal D_B\).
\end{proof}

Recall the notation  \(\Omega_h\) from the introduction and  set
\[
 \mathcal J(h)=T(\Omega_h),\qquad \mathcal V(h)=|\Omega_h|.
\]
For \(\varphi\in C^\infty(\S^{n-1})\) we write
\(T_\varphi(t)=\mathcal J(t\varphi)\) and
\(V_\varphi(t)=\mathcal V(t\varphi)\). Combining \eqref{eq:hadamard},
Theorem~\ref{thm:hessian} together with \eqref{eq:ball-state} and Lemma~\ref{lem:ball-large}
yields
\begin{equation}\label{eq:ball-hessian}
 T_\varphi'(0)=\lambda_*\int_{\S^{n-1}}\varphi\dHaus,
 \qquad
 T_\varphi''(0)
 =\lambda_*\int_{\S^{n-1}}
 \big(n\varphi^2-\varphi\Lambda\varphi\big)\dHaus.
\end{equation}
On the other hand, polar coordinates give
\begin{equation}\label{eq:ball-volume}
 \mathcal V(h)=\frac1n\int_{\S^{n-1}}(1+h)^n\dHaus,
\end{equation}
and hence
\begin{equation}\label{eq:ball-volume-deriv}
 V_\varphi'(0)=\int_{\S^{n-1}}\varphi\dHaus,\qquad
 V_\varphi''(0)=(n-1)\int_{\S^{n-1}}\varphi^2\dHaus.
\end{equation}

\begin{lemma}\label{lem:ball-exact-path}
Let \(t\mapsto h(t)\) be \(C^2\) with values in \(C^m(\S^{n-1})\) for
every \(m\), with \(h(0)=0\) and \(|\Omega_{h(t)}|=|B|\). Then
\(\varphi=\dot h(0)\) has zero mean and
\begin{equation}\label{eq:ball-Q}
 -\frac{d^2}{dt^2}T(\Omega_{h(t)})\Big|_{t=0}
 =\lambda_*\int_{\S^{n-1}}\big(\varphi\Lambda\varphi-\varphi^2\big)\dHaus.
\end{equation}
\end{lemma}

\begin{proof}
Differentiating \(\mathcal V(h(t))=|B|\) at \(t=0\) with
\eqref{eq:ball-volume} gives \(\int_{\S^{n-1}}\varphi=0\). We compare \(-T\) with the Lagrangian
\[
 \mathcal L(h)=-\mathcal J(h)+\lambda_*\big(\mathcal V(h)-|B|\big).
\]
 Along the path,
\(\mathcal L(h(t))=-T(\Omega_{h(t)})\), since the volume is fixed. We first
show that \(\mathcal L\) is of class \(C^2\) near \(0\) in
\(C^2(\S^{n-1})\). By \eqref{eq:ball-volume}, \(\mathcal V\) is a
polynomial in \(h\). For \(\mathcal J\), choose a smooth cutoff \(\chi\),
supported in \((-r_0,r_0)\) and equal to one near zero, and set
\(X_h(\rho\theta)=\chi(1-\rho)h(\theta)\theta\). The map \(h\mapsto X_h\) is
bounded and linear from \(C^2\) to \(W^{1,\infty}\), and
\(\mathrm{Id}+X_h\) is a bi-Lipschitz map of \(B\) onto \(\Omega_h\) for
small \(h\). The proof of Proposition~\ref{prop:bulk} applies with the
Banach parameter \(h\) in place of \(t\) and \(\mathrm{Id}+X_h\) in place
of \(\Phi_t\): it only uses that the deformation depends in a \(C^2\) way on
the parameter with values in \(W^{1,\infty}\), with uniform bi-Lipschitz
bounds. Hence \(\mathcal J\) is of class \(C^2\).

We now compute the derivatives of \(\mathcal L\) at \(0\) in the 
direction \(\varphi\). Set \[L(t)=\mathcal L(t\varphi)
=-T_\varphi(t)+\lambda_*(V_\varphi(t)-|B|).\] By \eqref{eq:ball-hessian}
and \eqref{eq:ball-volume-deriv},
\[
 L'(0)=0,\qquad
 L''(0)=-T_\varphi''(0)+\lambda_*V_\varphi''(0)
 =\lambda_*\int_{\S^{n-1}}\big(\varphi\Lambda\varphi-\varphi^2\big)\dHaus.
\]
Since \(\mathcal L\) is \(C^2\), \(D\mathcal L(0)[\varphi]=L'(0)\) and
\(D^2\mathcal L(0)[\varphi,\varphi]=L''(0)\). Finally, \(t\mapsto h(t)\) is \(C^2\) with values in \(C^2(\S^{n-1})\), so
the chain rule gives
\[
 \frac{d^2}{dt^2}\mathcal L(h(t))\Big|_{t=0}
 =D^2\mathcal L(0)[\varphi,\varphi]
 +D\mathcal L(0)[\ddot h(0)]
 =\lambda_*\int_{\S^{n-1}}\big(\varphi\Lambda\varphi-\varphi^2\big)\dHaus,
\]
because \(\ddot h(0)\) is smooth. Since \(\mathcal L(h(t))=-T(\Omega_{h(t)})\),
this is \eqref{eq:ball-Q}.
\end{proof}

\begin{proof}[Proof of Theorem~\ref{thm:ball-coercivity}]
Let \(h(t)\) be a smooth volume-preserving path, in the sense of
Lemma~\ref{lem:ball-exact-path}. By that lemma, \(\varphi=\dot h(0)\)
satisfies \eqref{eq:ball-volume-tangent}, and \eqref{Hessconstraint} holds
with \(\mathcal Q\) given by the first expression in \eqref{eq:constrained}.
Since \(\mathcal P\varphi_\ell(\rho\theta)=\rho^\ell\varphi_\ell(\theta)\),
we have \(\Lambda\varphi_\ell=\ell\varphi_\ell\), and orthogonality gives
the second expression in \eqref{eq:constrained}. Conditions
\eqref{eq:ball-volume-tangent} and \eqref{eq:ball-centering} are precisely
\(\varphi_0=\varphi_1=0\), and
\[
 \frac13(1+\ell)\le\ell-1\le1+\ell\qquad(\ell\ge2).
\]
This proves \eqref{eq:ball-coercivity}, with equality in the lower bound
exactly on harmonics of degree two. The upper constant is also optimal,
since \((\ell-1)/(\ell+1)\to1\).
\end{proof}

\section*{Acknowledgments}
The authors acknowledge financial support from the ANR grant STOIQUES.
MN is member of INdAM/GNAMPA. BR is a member of INdAM--GNAMPA and has been partially supported by the ``INdAM - GNAMPA Project'', code CUP E53C25002010001.

\medskip
\noindent\textbf{Use of AI.}
The proofs in this paper were found mostly with the help of a large language
model (LLM), which was also used to help write up the paper. The work involved
substantial interaction between the authors and the model. The authors have
checked the results and the text, and take full responsibility for the content
of the paper.

\bibliographystyle{acm}
\bibliography{biblio}

\end{document}